\documentclass[10pt]{article}
\usepackage{amssymb,amsmath,amsthm}
\usepackage[T1]{fontenc}
\usepackage{fullpage}
\usepackage{graphicx}
\usepackage{color}
\newtheorem{theorem}{Theorem}
\newtheorem{definition}[theorem]{Definition}

\newtheorem{proposition}[theorem]{Proposition}

\title{On the harmonious chromatic number of graphs}
\author{Gabriela Araujo-Pardo\footnotemark[1] \and Juan Jos{\' e} Montellano-Ballesteros\footnotemark[1] \and Mika Olsen\footnotemark[2] \and Christian Rubio-Montiel\footnotemark[3]}

\begin{document}
\maketitle

\def\thefootnote{\fnsymbol{footnote}}
\footnotetext[1]{Instituto de Matem{\'a}ticas, Universidad Nacional Aut{\'o}noma de M{\' e}xico, Mexico City, Mexico. {\tt [garaujo|juancho]@math.unam.mx}.}
\footnotetext[2]{Departamento de Matem{\' a}ticas Aplicadas y Sistemas, UAM-Cuajimalpa, Mexico City, Mexico. {\tt
olsen@correo.cua.uam.mx}.}
\footnotetext[3]{Divisi{\' o}n de Matem{\' a}ticas e Ingenier{\' i}a, FES Acatl{\' a}n, Universidad Nacional Aut{\'o}noma de M{\' e}xico, Naucalpan, Mexico. {\tt christian.rubio@acatlan.unam.mx}.}

\begin{abstract} 
The harmonious chromatic number of a graph $G$ is the minimum number of colors that can be assigned to the vertices of $G$ in a proper way such that any two distinct edges have different color pairs. This paper gives various results on harmonious chromatic number related to homomorphisms, incidence graphs of finite linear systems, and some circulant graphs.
\end{abstract}



%%%%%%%%%%%%%%%%%%%%%%%%%%%%%%%%%%%%%%%%%%%%%%%%%%%%%%%%%%%%%%%%%%

\section{Introduction}

A \emph{harmonious $k$-coloring} of a finite graph $G$ is a proper vertex $k$-coloring such that every pair of colors appears on at most one edge.  The \emph{harmonious chromatic number} $h(G)$ of $G$ is the minimum number $k$ such that $G$ has a harmonious $k$-coloring. % harmonious chromatic number has applications in codes, usamos aspectso geometricos y algebraicos de diseos, 

The harmonious chromatic number introduced by Miller and Pritikin \cite{MR1139582} is a proper variation of the parameter defined independently by Hopcroft and Krishnamoorthy \cite{MR711339} and by Frank, Harary, and Plantholt \cite{MR683991}.

From the definition, a graph $G$ has the size $m \leq \binom{h(G)}{2}$, which in turn gives the easy lower bound
\begin{equation}\label{eq1}
\sqrt{2m+\frac{1}{4}}+\frac{1}{2} \leq h(G). 
\end{equation}
In \cite{MR989131} and \cite{ORA} the authors %determine $h(C_n)$ as well as $h(C_r\cup C_s)$ for $r+s=n$, 
give families of graphs that show that the lower bound (\ref{eq1}) is tight.

It is easy to see that a proper $n$-coloring of a graph $G$ of order $n$ is harmonious, however,  the harmonious chromatic number has an additional difficulty that makes impossible a greedy coloring to complete a partial coloration even if extra colors are allowed: if we have an uncolored vertex adjacent to two different vertices with the same color, there is no harmonious extension of the involved partial coloring.

Therefore, if $G$ has order $n$ and diameter at most two, then $h(G)=n$, see Edwards \cite{MR1477743}, for instance, complete graphs, complete bipartite graphs, line graphs of complete graphs, Erd{\H o}s-R{\' e}nyi orthogonal polarity graphs, graphs with a universal vertex (stars, wheels, connected trivially perfect graphs, etc.) have the harmonious number equal to $n$. In the survey \cite{MR1477743}, we can find a list of interesting results about the harmonious chromatic number and the achromatic number, a related parameter, also see Chapters 12 and 13 of \cite{MR2450569}. 

The purpose of this paper is to study the harmonious chromatic number% this parameter
. The paper is organized as follows. In Section \ref{section2} are given results arising from homomorphisms; in Section \ref{section3} we study the harmonious chromatic number of graphs with small diameter related to incidence graphs of linear systems; and finally, in Section \ref{section4} we analyze the harmonious chromatic number of some circulant graphs.

\section{On homomorphisms}\label{section2}

Let $G$, $H$ be graphs. A mapping $\varsigma\colon V(G) \rightarrow V(H)$ is a \emph{homomorphism} from $G$ to $H$ provided that $\varsigma (x) \varsigma (y) \in E(H)$ whenever $xy\in E(G)$. Then the graph $\varsigma (G) = (\varsigma (V(G)),\varsigma_E(E(G)))$, where
\[\varsigma_E(E(G)) = \{\varsigma (x)\varsigma (y) \colon xy \in E(G)\},\]
is the \emph{homomorphic image} of $G$ in $H$. An \emph{elementary homomorphism} of the graph $G$ is a homomorphism that identifies nonadjacent vertices $u,v\in V (G)$, i.e., maps both $u$, $v$ to a vertex out of $V (G) \setminus \{u, v\}$, and fixes all vertices of $V (G) \setminus \{u, v\}$; in this paper we suppose that the elementary homomorphism identifying $u$ and $v$ maps both $u$, $v$ to the vertex $u$. An \emph{elementary harmonious homomorphism} of the graph $G$ is an elementary homomorphism of $G$ identifying vertices $u$, $v$ with $d_G(u, v) \geq 3$.

\begin{proposition}\label{propo1}
If $\epsilon$ is an elementary harmonious homomorphism of a graph $G$, then $h(G) \leq h(\epsilon (G))$.
\end{proposition}
\begin{proof}
Suppose that $\epsilon$ identifies vertices $u$, $v$ with $d_G(u, v) \geq 3$, $h(\epsilon (G)) = k$, and let $\varsigma'$ be a harmonious $k$-coloring of $\epsilon (G)$. Consider a vertex $k$-coloring $\varsigma$ of $G$ determined by $\varsigma'$ as follows:
\[\varsigma(x)= \begin{cases}
\begin{array}{ll}
\varsigma'(x) & \textrm{ if }x\in V(G)\setminus\{u,v\}\\
\varsigma'(u) & \textrm{ if }x\in\{u,v\}
\end{array}\end{cases}\] 
Clearly, $\varsigma$ is harmonious, and so $h(G) \leq k = h(\epsilon (G))$.
\end{proof}

A multiple application of elementary harmonious homomorphisms leads from the graph $G$ to a graph $H$ of diameter two. The composition of those homomorphisms is a homomorphism $\varsigma$ from $G$ to $H$. If $V(H)=\{v_1, v_2, \dots , v_k\}$, then $\varsigma$ corresponds to a partition $\{V_1, V_2, \dots , V_k\}$ of $V (G)$, where, for each $i\in \{1,2,\dots k\}$, the set $V_i = \varsigma^{-1}(v_i)$ is formed by vertices that are (in $G$) pairwise at a distance at least three. Moreover, if $i,j\in \{1, 2, \dots, k\}$, $i\not = j$, there is at most one edge in $E(G)$ joining a vertex of $V_i$ to a vertex of $V_j$. To see it, suppose that there are edges $xa, yb in E(G)$ with $x, y \in V_i$, $x\not= y$, and $a, b \in V_j$, $a\not = b$. Without loss of generality we may suppose that, during the process of determining $\varsigma$, the identification of $a$ and $b$ took place before the identification of $x$ and $y$. Then, however, in the graph, that is created by identifying $a$ and $b$, the distance between $x$ and y is two. Since this distance cannot increase subsequently, the identification of $x$ and $y$ cannot be carried out, a contradiction. The above reasoning shows that the homomorphism $\varsigma \colon V (G) \rightarrow V (H)$ is a harmonious coloring of $G$; we say $\varsigma$ is a \emph{harmonious homomorphism} from $G$ to $H$ (or simply a harmonious homomorphism of the graph $G$). The following result is a consequence of Proposition \ref{propo1}.

\begin{theorem}\label{teo2}
Let $u$ and $v$ be vertices in a graph $G$ such that $d(u,v)\geq 3$ and let $\epsilon$ be the elementary harmonious homomorphism of $G$ that identifies the vertices $u$ and $v$. Then
\[h(\epsilon(G)) = h(G)\]
if and only if there exists an $h(G)$-coloring of $G$ assigning to $u$ and $v$ the same color.
\end{theorem}
\begin{proof}
First, suppose that $h(\epsilon (G)) = h(G) = k$, and let $\varsigma'$ be a harmonious $k$-coloring of $\epsilon (G)$. Then the coloring $\varsigma$, which is determined by $\varsigma'$ as in the proof of Proposition \ref{propo1}, is a harmonious $h(G)$-coloring of $G$ with $\varsigma(u) = \varsigma (v)$.
Now, suppose that $\varsigma(u) = \varsigma(v)$ for a harmonious $h(G)$-coloring $\varsigma$ of $G$. Then the restriction of $\varsigma$ to $V (\epsilon (G))$ is a harmonious $h(G)$-coloring of $\varsigma (G)$. Thus $h(\varsigma (G)) \leq h(G)$. By Proposition \ref{propo1}, $h(G) \leq h(\varsigma (G))$, and the result follows.
\end{proof}

\begin{theorem}\label{teo3}
If $\epsilon$ is an elementary harmonious homomorphism identifying nonadjacent vertices $u$ and $v$ of $G$ such that $d(u,v)\geq 3$, then 
\[h(G) \leq h(\epsilon (G)) \leq h(G) + \min\{\deg(u),\deg(v)\} + 1.\]
Therefore, $h(G) \leq h(\epsilon (G)) \leq h(G) + \Delta(G) + 1.$
\end{theorem}
\begin{proof}
By Proposition \ref{propo1}, $h(G)\leq h(\epsilon (G))$.  Suppose that $h(G) = k$ and a harmonious $k$-coloring of $G$ is given with the colors $1, 2, \dots , k$ and $\deg(u)=\min\{\deg(u),\deg(v)\}$ with $N(u)=\{u_1,u_2,\dots ,u_{\deg(u)}\}$. Define a coloring $\varsigma'$ of $\epsilon (G)$ by
\[\varsigma'(x)= \begin{cases}
\begin{array}{ll}
\varsigma(x) & \textrm{ if }x\in V(G)\setminus N(u)\setminus\{u,v\}\\
k+1 & \textrm{ if }x\in\{u,v\}\\
k+i+1 & \textrm{ if }x=u_{i}
\end{array}\end{cases}.\]

Since $\varsigma'$ is a harmonious coloring of $\epsilon (G)$ using at most $k + \deg(u) + 1$ colors, the result follows.
\end{proof}

\begin{theorem}\label{teo4}
Let $G$ be a graph with vertices $u$ and $v$ in $G$ such that $d_G(u,v)\geq 3$. Then
\[h(G) \leq h(G + uv) \leq h(G)+1.\]
\end{theorem}
\begin{proof}
Clearly, $h(G) \leq h(G + uv)$. To prove the other inequality, consider a harmonious $h(G)$-coloring $\varsigma$ of $G$. If $\varsigma$ is a harmonious coloring of $G + uv$, then $h(G + uv) = h(G)$. Otherwise recoloring one of vertices $u$, $v$ with a new color yields a harmonious $(h(G) + 1)$-coloring of $G + uv$. Indeed, if $\varsigma (u) = \varsigma (v)$, then recoloring $u$ does the job. On the other hand, under the assumption $\varsigma (u)\not = \varsigma (v)$ there is $w \in \{u, v\}$ such that with $\{w, x\} = \{u, v\}$ the vertex $w$ does not have in $G$ a neighbor colored with $\varsigma (x)$; in such a case recolor $w$.
\end{proof}

Beginning with a graph $G$, we can always perform a sequence of elementary harmonious homomorphisms until arriving at some graph of diameter two $H$ with $k$ vertices. As we saw,  the graph $H$ obtained in this manner corresponds to a harmonious $k$-coloring of $G$. Consequently, we have the following.

\begin{definition}
The largest order of a graph of diameter two that is a harmonious homomorphic image of a graph $G$ is the harmonious achromatic number of $G$, denoted by $ha(G)$.
\end{definition}

We have seen that for every graph $G$ of order $n$, \[h(G) \leq ha(G) \leq n.\]

A \emph{complete coloring} of $G$ is a proper $k$-coloring such that every pair of colors appears on at least one edge. The \emph{achromatic number} $\alpha(G)$ of $G$ is the maximum number $k$ such that $G$ has a complete $k$-coloring.

In the case of the chromatic and the achromatic number there is an interpolation result \cite{MR0272662}, namely,  a graph $G$ has a complete proper $k$-coloring if and only if $\chi(G) \leq k \leq \alpha(G)$. By giving a simple example, we show that this interpolation result fails in the case of the harmonious chromatic number and the harmonious achromatic number. 

\begin{theorem}\label{teo6}
Let $r \geq 4$ be a positive integer. Then there exists a graph $G_r$ such that $h(G_r) = r+1$, $ha(G_r) = 2r-1$, and $h(\epsilon (G_r)) \in \{r+1,2r-1\}$ whenever $\epsilon$ is a harmonious homomorphism of the graph $G_r$.
\end{theorem}
\begin{proof}
Let $G_r$ be the vertex-disjoint union of graphs $L = K_{1,r-1}$ and $R = K_{1,r-1}$ with $V(L) = \{u_0,u_1,\dots ,u_{r-1}$ and $V(R) = \{v_0,v_1,\dots ,v_{r-1}$, where $u_0$, $v_0$ are vertices of degree $r-1$. It is easy to see that $h(G_r)=r+1$.

Now let $f$ be a harmonious homomorphism of the graph $G_r$, and let $\epsilon$ be the first elementary harmonious homomorphism applied when determining $f$. The homomorphism $\epsilon$ identifies vertices $u$, $v$ with $d_{G_r} (u, v)\geq 3$, say $u \in V (L)$ and $v \in V (R)$.

\begin{enumerate}
\item If $u=u_0$ and $v=v_0$,then $\epsilon (G_r)=f(G_r)=K_{1,2r-2}$ and $h(f(G_r))= 2r-1$.

\item If $u = u_0$ and $v\not = v_0$, then without loss of generality we may suppose that $v = v_{r-1}$. Any other elementary harmonious homomorphism applied in the process of establishing $f$ has to identify a leaf of $L$ with a leaf $v_j$ of $R$, $j\not = r-1$. Without loss of generality let $u_i$ be identified with $v_i$ for $i\in\{ 1,2,\dots ,r-2\}$. The resulting graph $f(G_r)$ of diameter two has the leaf $u_{r-1}$, $r-2$ vertices $u_i$ of degree two, $i = \{1,2,\dots,r-2$, the vertex $u_0$ of degree $r$ and its neighbor $v_0$ of degree $r-1$; therefore, $h(f(G_r)) = r + 1$.

\item If $u\not = u_0$ and $v = v_0$, we proceed similarly as in the case 2 to conclude that $h(f(G_r)) = r + 1$.

\item If $u\not = u_0$ and $v\not = v_0$, then $d_{\epsilon (G_r)}(u_0,v_0) = 2$. Therefore, at most one of the vertices $u_0$, $v_0$ is involved in elementary harmonious homomorphisms leading to the resulting graph $H$ of diameter two. Consequently, $H$ is either (isomorphic to) the graph $f(G_r)$ from the case $c\in \{2,3\}$ or equal to $K_{2,r-1}$, and so we have $h(H) = r + 1$.
\end{enumerate}
By inspection of the cases 1-4 we see that $ha(G_r) = 2r-1$.
\end{proof}

%The complement $\overline{G}$ of a graph $G$ is the graph with vertex set $V(G)$ and two of its vertices $u$ and $v$ are adjacent if and only if $u$ is not adjacent to $v$ in $G$.

\begin{theorem}
If $\epsilon$ is an elementary harmonious homomorphism of a graph $G$, then \[h(\overline{G})-1 \leq h(\overline{\epsilon(G)}) \leq h(\overline{G}).\]
\end{theorem}
\begin{proof}
Suppose that the homomorphism $\epsilon$ identifies vertices $u,v\in V (G)$ (with $d_G(u, v) \geq 3$). 

To prove that $h(\overline{G})-1\leq h(\overline{\epsilon (G)})$ let $h(\overline{\epsilon(G)}) = k$, and let a coloring $\varsigma\colon V (\overline{\epsilon(G)}) \rightarrow \{1, 2,\dots , k\}$ be harmonious. Consider the proper vertex coloring $\varsigma' \colon V(\overline{G})\rightarrow \{1,2,\dots ,k+1\}$ defined by
\[\varsigma'(x)= \begin{cases}
\begin{array}{ll}
\varsigma(x) & \textrm{ if }x\in V(\overline{G})\setminus \{v\}\\
k+1 & \textrm{ if }x=v
\end{array}\end{cases}.\]

Evidently, $\varsigma'$ is harmonious. Therefore, $h(G) \leq k + 1 = h(\varsigma (G)) + 1$, and the required inequality follows.

Now, we show that $h(\overline{\epsilon (G)}) \leq h(\overline{G})$. With $h(\overline{G}) = l$ let $\mathcal{P} =\{V_1,V_2,\dots,V_l\}$ be the partition of $V (\overline{G})$ into (independent) color classes of a harmonious $l$-coloring of $\overline{G}$. Since $uv \not\in E(G)$, we have $uv \in E(\overline{G})$, and so without loss
of generality we may suppose that $u\in V_1$ and $v\in V_l$. Let $V'_i =V_i$ for $i\in \{1,2,\dots,l-1\}$, and $V'_l = V_l\setminus \{v\}$. Further, let $p = l$ if $V'_l\not = \emptyset$, and $p = l-1$ if $V'_l = \emptyset$. Then $\mathcal{P}' = \{V'_1,V'_2,\dots,V'_p\}$ is a partition of $V(\overline{\epsilon(G)})$. The sets $V'_2,V'_3,\dots ,V'_p$ do not contain $u$, hence they are independent in $\overline{\epsilon (G)}$. The set $V'_1$ is independent in $\overline{\epsilon (G)}$, too. First, distinct vertices of $V'_1\setminus \{u\}$ are nonadjacent in $\overline{\epsilon (G)}$, since they are nonadjacent in $\overline{G}$. Next, if vertices $u$ and $w \in V'_1\setminus \{u\}$ are adjacent in $\epsilon (G)$, they are nonadjacent in $\epsilon (G)$, which means that neither $uw$ nor $vw$ is an edge in $G$; therefore, vertices $u, w \in V_1$ are adjacent in $\overline{G}$, a contradiction. The partition $\mathcal{P}'$ represents a harmonious coloring of $\overline{\epsilon (G)}$. Indeed, it suffices to prove that for any $i \in \{2,3,\dots ,p\}$ there is at most one edge joining in $\overline{\epsilon (G)}$ a vertex of $V'_1$ to a vertex of $V'_i$. Assuming the opposite, there are vertices $w \in V'_1\setminus\{u\}$ and $x,y\in V'_i$, $x\not= y$, such that $ux, wy \in E(\overline{\epsilon (G)})$. Consequently, $\{ux, wy\} \cap E(\epsilon (G)) = \emptyset$, which implies that $ux, vx, wy \notin E(\overline{G})$. Then, however, $ux, wy \in E(\overline{G})$, $u, w \in V_1$ and $x,y \in V_i$, and we have obtained a contradiction with the fact that $\mathcal{P}$ represents a harmonious coloring of $\overline{G}$. Thus $h(\overline{\epsilon (G)}) \leq p \leq l = h(\overline{G})$.
\end{proof}

Since $ha(G)\leq n$ for any graph $G$ of order $n$, we deduce that \[ha(G)+ha(\overline{G})\leq 2n.\]
In Kundrík \cite{MR1206261} it was proved that 
\[n+1 \leq h(G)+h(\overline{G})\leq 2n.\]
Moreover, for any $n\geq 7$ and $r\in \{1,2,\dots,n\}$,  there exists a graph $G$ of order $n$ such that $h(G)+h(\overline{G})=n+r$.

%%%%%%%%%%%%%%%%%%%%%%%%%%%%%%%%%%%%%%%%%%%%%%%%%%%%%%%%%%%%%%%%%%

\section{On graphs of small diameter}\label{section3}

Consider a graph $H$ of order $n$ and diameter two such that deleting any edge increases the diameter; such graph is called a \emph{diameter-two-critical} graph.  Clearly, $G=H-e$ for any edge $e$ of $H$ is a graph of diameter three or more with $h(G)=n-1$. The following theorem states there is no constant $c$ such that, for any graph $G$ of order $n$ and diameter three, $n-h(G) \leq c$. We recall that the join of graphs $G,H$ is the graph $G+H$ with $V(G+H)$ equal to the vertex-disjoint union of $V(G), V (H)$, and $E(G+H) = E(G)\cup E(H) \cup \{uv \colon u \in V (G), v \in V (H)\}$.

\begin{theorem} \label{teo8}
Let $r, s\geq 3$ be positive integers. There exists a graph $G$ of order $n$ and diameter 3 such that $n=r(s+1)$ and $h(G)=r+s$.
In particular, if $r=s$, then $h(G)=2\sqrt{n}+\Theta(1)$.
\end{theorem}
\begin{proof}
Let $G_{r,s}$ be the graph obtained by adding $s$ leaves to each vertex of the complete graph $K_r$. A vertex in the complete subgraph $K_r$ has degree $r+s-1$ in $G_{r,s}$, thus $h(G_{r,s})\geq r+s$. So, the diameter of $G_{r,s}$ is three. Moreover, there exists a harmonious homomorphism from $G_{r,s}$ to the graph $K_r+\overline{K_s}$ of order $r + s$ and diameter two, hence $h(G_{r,s}) \leq r + s$.
\end{proof}

Now we analyze graphs arising from finite geometries.  A finite projective plane of order $q$ has $q^2+q+ 1$ points and $q^2+q+ 1$ lines satisfying the following axioms.
\begin{enumerate}
\item Any two distinct points determine a line.
\item Any two distinct lines determine a point.
\item There exists four points such that no line is incident to three or four of them.
\end{enumerate}

In particular, if $q$ is a power of a prime number there exists a projective plane that arises from the Galois field of order $q$,  we denote the field by $GF(q)=\{g_0=0,g_1=1,g_2,\dots,g_{q-1}\}$. It is called the \emph{algebraic projective plane}, and it is denoted by $PG(2,q)$.  For further information about the algebraic projective planes,  see \cite{MR554919}.  Next, we take a projective plane $PG(2,q)$ and we will label their points and lines. 

%Without loss of generality, 
Let $P$ be a point and let $L$ be a line incident to $P$.  We use $P$ and $L$ to label the elements of the projective plane. We call $P$ the infinity point and $L$ the infinity line. Since each line has $q+1$ points, let $\{P_{g_0},P_{g_1},\ldots,P_{g_{q-1}}\}$ be the set of points incident to $L$ that are different from $P$. Since each point is incident to $q+1$ lines, let $\{L_{g_0},L_{g_1},\ldots,L_{g_{q-1}}\}$ be the set of lines incident to $P$ that are different from $L$.  
Let $\{(g_i,g_0),(g_i,g_1),\ldots,(g_i,g_{q-1})\}$ be the set of points different from $P$ incident to $L_{g_i}$. The remaining lines are labeled as follows. The line $[a,b]$ is incident to all the points $(x,y)$ that satisfy $y=ax+b$ using the arithmetic of $GF(q)$, for all $a,b,x,y\in GF(q)$.

Now, we recall the definition of the incidence graph of a projective plane of order $q$, that is, a regular bipartite graph of degree $q+ 1$ and bipartition $(A,B)$ with $2(q^2+q+ 1)$ vertices: the set $A$ have $q^2+q+1$ vertices corresponding to the points and the set $B$ have $q^2+q+1$ vertices corresponding to the lines. Two vertices of the incidence graph are adjacent if the involved point and the involved line are incident to each other. That graph is of diameter three, girth six, and it is a $(q+ 1,6)$-cage which has the order equal to the well-known Moore bound. The incidence graph is also known as the Levi graph.

\begin{theorem}\label{teo9}
If $G$ is the incidence graph of the algebraic projective plane of odd order $q$, then $q^2+q+ 1\leq h(G) \leq q^2+q+ 2$.
\end{theorem}
\begin{proof}
In any harmonious coloring of $G$, different points receive different colors, because every two points share a line, then  $h(G)\geq q^2+q+ 1$.

We define the coloring in the points and lines as follows.  We assign $q^2+q+1$ distinct colors to $q^2+q+1$ points of $PG(2,q)$. Now, if $a\not= 0$, the line $[a,b]$ is colored with the color of the point $(a,b)$. Since $a^2\not= 0$ implies $b\not = a^2 + b$, $(a, b)$ is not incident to $[a, b]$. Then the partial coloring (defined at that moment) is proper. Moreover,  for $\alpha\not=0$, if the point $(\alpha,\beta)$ is incident to the line $[a,b]$, $\beta=a\alpha+b$,  then $(\alpha,\beta)\not =(a,b)$, and the point $(a,b)$ is not incident to the line $[\alpha,\beta]$, otherwise, $b=a\alpha+\beta=a\alpha+a\alpha+b$, i.e., $0=1+1$, a contradiction because $q$ is odd. 
Since for any pair of color classes $\{(a,b),[a,b]\}$ and $\{(m,n),[m,n]\}$ with $a\not=0$ and
$m\not= 0$ there is at most one edge joining them, the partial coloring is harmonious.

For $b\not=-1$, the line $[0,b]$ is colored with the color of the point $(0,b+1)$,  the line $[0,-1]$ is colored with the color of the point $P$, and the line $L$ is colored with the color of the point $(0,0)$. Clearly, the partial coloring is proper. If a point $(0,b+1)$ is incident to the line $[a,b+1]$ then the point $(a,b+1)$ is not incident to the line $[0,b]$ since $b+1\not=b$. Therefore, the partial coloring is harmonious.

For $a\not=0,-1$,  the line $L_a$ is colored with the color of the point $P_{a+1}$, the line $L_{-1}$ is colored with the color of the point $P_{1}$, and the line $L_{0}$ is colored with a new color.  Clearly, the coloring is proper. Now, the point $P_1$ is incident to the line $[1,b]$ and $(1,b)$ is not incident to the line $L_{-1}$.  If a point $P_{a+1}$ is incident to the line $[a+1,b]$, then the point $(a+1,b)$ is not incident to the line $L_a$ since $a+1\not=0$.  Finally, the point $P$ is incident to the line $L_a$, the point $P_{a+1}$ is not incident to the line $[0,-1]$, and the point $P_1$ is incident to the line $L$, and the point $(0,0)$ is not incident to the line $L_{-1}$. Therefore, the coloring is harmonious and $h(G)\leq q^2+q+ 2$.
\end{proof}

A linear space $\mathcal{S}$ is an incidence structure satisfying the following axioms:
\begin{enumerate}
\item Any two distinct points determine a line.
\item Any line is incident to at least two distinct points.
\item $\mathcal{S}$ contains at least two distinct lines.
\end{enumerate}
Consequently, using the same techniques as in Theorem \ref{teo9}, we can give bounds on the harmonious number of the incidence graph of a linear space. 

Examples of linear spaces are projective planes,  affine planes, and $2$-$(n,k)$-designs (see the definition below).  Clearly, the incidence graph of a linear space has diameter at most 4.

Now, we give a bound for the harmonious chromatic number of the incidence graph of a linear space related to the Erd{\H o}s-Faber-Lov{\' a}sz Conjecture, see \cite{MR4249208,MR0409246,MR602413}. 

A \emph{line coloring} of $\mathcal{S}$ is an assignment of $k$ colors to lines of $\mathcal{S}$. A line coloring of $\mathcal{S}$ is called \emph{proper} if any two intersecting lines have different colors. The \emph{chromatic index} $\chi'(\mathcal{S})$ of $\mathcal{S}$ is the smallest $k$ such that there exists a proper line coloring of $\mathcal{S}$ with $k$ colors. Erd{\H o}s, Faber and Lov{\' a}sz conjectured that the chromatic index of any finite linear space $\mathcal{S}$ cannot exceed the number of its points, so if $\mathcal{S}$ has $n$ points then\[\chi'(\mathcal{S})\leq n.\]

\begin{theorem} \label{teo10}
Let $\mathcal{S}$ be a linear space of $n$ points, then its incidence graph $G$ has the following bounds for $h(G)$,
\[n\leq h(G) \leq n+\chi'(\mathcal{S}).\]
Moreover, if the Erd{\H o}s-Faber-Lov{\' a}sz Conjecture is true for $\mathcal{S}$, then $h(G) \leq 2n$.
\end{theorem}
\begin{proof}
The incidence graph $G$ of a linear space with point set $\mathcal{P}$ has $h(G)\geq |\mathcal{P}|=n$ by axiom 1. 
To get the upper bound, we assign a coloring via a line coloring with $\chi'(\mathcal{S})$ colors and assign a different color to each vertex and the result follows because any two lines with the same color are at distance 3 or 4 in $G$.
\end{proof}

Let $n$, $b$, $k$ and $r$ be positive integers with $n>1$. Let $D=(P,B,I)$ be a triple consisting of a set $P$ of $n$ distinct objects, called points of $D$, a set $B$ of $b$ distinct objects, called blocks of $D$ ($P\cap B = \emptyset$), and an incidence relation $I$, a subset of $P\times B$. We say that $v$ is incident to $u$ if exactly one of the ordered pairs $(u,v)$ and $(v,u)$ is in $I$; then $v$ is incident to $u$ if and only if $u$ is incident to $v$.  $D$ is called a \emph{$2$-$(n, b, k, r)$ block design} (for short, a \emph{$2$-$(n, b, k, r)$ design}) if it satisfies the following axioms.
\begin{enumerate}
\item Each block of $D$ is incident to exactly $k$ distinct points of $D$.
\item Each point of $D$ is incident to exactly $r$ distinct blocks of $D$.
\item If $u$ and $v$ are distinct points of $D$, then there is exactly one block of $D$ incident to both $u$ and $v$.
\end{enumerate}
A $2$-$(n, b, k, r)$ design is called a balanced incomplete block design BIBD; it is called an $(n, k)$-design, too, since the parameters of a $2$-$(n, b, k, r)$ design are not all independent. The two equations connecting them are $nr=bk$ and $r(k-1)=n-1$. For a detailed introduction to block designs we refer the reader to \cite{MR1729456}.

A design is \emph{resolvable} if its blocks can be partitioned into $r$ sets so that $b/r$ blocks of each part are point-disjoint and each part is called a \emph{parallel class}.
Therefore,  Theorem \ref{teo10} implies that the incidence graph $G$ of a resolvable design $D$ satisfies $h(G)\leq n+\frac{n-1}{k-1}$.

When $k=2$, the $(n,k)$-design is isomorphic to the complete graph $K_n$ and it is resolvable if and only if $n$ is even. The following theorem improves the previous bounds.

\begin{theorem}
Let $G$ be the incidence graph of the complete graph $K_n$. Then 
\[\frac{3}{2}n\leq h(G) \leq \frac{5n+10}{3}n.\]
\end{theorem}
\begin{proof}
For the lower bound,  let $G$ be the incidence graph of the complete graph $K_n$,  with $ V(K_n) = \{x_1, \dots, x_n\}$ and 
 let $\varsigma : E(K_n) \cup V(K_n) \rightarrow \{c_1,\dots, c_{h(G)}\}$ be a harmonious coloring.

By Theorem \ref{teo10}, $n\leq h(G)$. For each $x_i\in V(K_n)$, let $c_i = \varsigma(x_i)$. For each $i\in \{1,\dots,h(G)\}$,  let $e_i = \{e\in E(K_n) : \varsigma (e)= c_i\}$.  Observe that each $e_i$ induces a matching in $K_n$, $\sum\limits_{i=1}^{h(G)} |e_i| =\binom{n}{2}$ and since $\varsigma$ is a harmonious coloring, for every $x_i\in V(K_n)$, $x_i$  is not incident to an edge of $e_i$. 

Let $G^*$ be the spanning subgraph of $K_n$ such that for each $xy \in E(G^*)$, $\varsigma (xy) \in \{c_1, \dots, c_n\}$.  On the one hand, $e(G^*) = \sum\limits_{i=1}^{n} |e_i| = \frac{1}{2}\sum\limits_{i=1}^{n} \deg_{G^*} (x_i)$. On the other hand, for each  $x_i \in V(G^*)$, if there is $x_ix_j  \in E(G^*)$ such that $\varsigma(x_ix_j) = c_k$, then $x_k$ is not incident to an edge of $e_i$. Thus, $|e_i| \leq \frac{n-1-\deg_{G^*}(x_i)}{2}$.

Therefore $\frac{1}{2}\sum\limits_{i=1}^{n} \deg_{G^*} (x_i) =\sum\limits_{i=1}^{n} |e_i|  \leq \sum\limits_{i=1}^{n} \frac{n-1-\deg_{G^*}(x_i)}{2}$ which implies that $\sum\limits_{i=1}^{n} \deg_{G^*} (x_i) \leq \binom{n}{2}$. Thus, $e(G^*) \leq \binom{n}{2}/2$ and therefore $\sum\limits_{j=n+1}^{h(G)} |e_j| \geq  \frac{\binom{n}{2}}{2}$. Since, for each $i\in \{1,\dots,h(G)\}$, $|e_i|\leq \lfloor\frac{n}{2}\rfloor$ it follows that $(h(G)-n)\lfloor\frac{n}{2}\rfloor \geq \binom{n}{2}/2$. Finally we have \[h(G)\geq\begin{cases}
\begin{array}{c}
\frac{3}{2}n-\frac{1}{2}\\
\frac{3}{2}n
\end{array} & \begin{array}{c}
\textrm{if }n\textrm{ is even}\\
\textrm{otherwise}
\end{array}\end{cases}\] and from here the result follows.

For the upper bound suppose that $n=3m -l \geq 3$, where $m,l$ are integers and $0\leq l \leq 2$. Let $l_0,l_1,l_2$ be integers with $0=l_0 \leq l_1 \leq l_2 \leq 1$ such that $l = l_0 + l_1 + l_2$. Consider a partition $\{V_0, V_1, V_2\}$ of $V (K_n)$ with $|V_i| = m - l_i$, $i\in\{0,1,2\}$, and a partition $\{C_0,C_1,C_2\}$ of the set $\{1,\dots 3m\}$ with $C_0 = \{1,\dots ,m\}$, $C_1 = \{m + 1,\dots ,2m\}$ and $C_2 = \{2m + 1,\dots ,3m\}$. Since $l_i \in \{0, 1\}$, we have $\chi '(K_n[V_i]) \leq \chi (K_n[V_i]) = m - l_i \leq m$, and there is a proper edge coloring $\varsigma_i'\colon E(K_n[V_i]) \rightarrow C_{i+1\mod 3}$, $i\in\{0, 1, 2\}$. The graph $G' = \overline{K_n[V_0] \cup K_n[V_1] \cup K_n[V_2]} \cong K_{m-l_0,m-l_1,m-l_2}$ satisfies $\Delta (G') = 2m - l_1$, hence, by Vizing's Theorem, there is a proper edge coloring $\varsigma '\colon E(G') \rightarrow \{3m + 1,\dots ,5m - l_1 + 1\}$. Now define a coloring $\varsigma \colon V(G_n) = V(K_n) \cup E(K_n) \rightarrow \{1, \dots , 5m - l_1 + 1\}$ as follows:

\begin{enumerate}
\item The restriction of $\varsigma$ to $V_i$ is an injection to $C_i$, $i\in\{0, 1, 2\}$.

\item The restriction of $\varsigma$ to $E(K_n[V_i])$ is $\varsigma_i'$, $i\in\{0, 1, 2\}$.

\item The restriction of $\varsigma$ to $E(G')$ is $\varsigma'$.
\end{enumerate}

It is not hard to see that $\varsigma$ is a harmonious coloring of $G$, hence $h(G) \leq 5m-l_1 +1$. From $n = 3m-l$ we obtain $m = \frac{n+l}{3}$, and then $h(G) \leq \frac{5n+5l-3l_1+3}{3} = \frac{5n+2l_1+5l_2+3}{3} \leq \frac{5n+10}{3}$.
\end{proof}

Next, we recall that a set $U \subset V(G)$ is a \emph{packing set} of $G$ if every pair of distinct vertices $u,v\in  U$ satisfies $d_G(u,v) \geq 3$. The \emph{packing number} $\rho (G)$ of $G$ is the maximum cardinality of a packing set of $G$.

\begin{theorem}
If $G$ is a graph of order $n$, then \[\frac{n}{\rho(G)}\leq h(G) \leq n-\rho (G)+1.\]
\end{theorem}
\begin{proof}
Suppose that $h(G) = k$ and let there be given a harmonious coloring of $G$ with $k$ colors and color classes $V_1, V_2, \dots , V_k$. Since $n=\sum_{i=1}^k|V_{i}|\leq k\rho(G)$, the lower bound follows.

Now,  let $U$ be a maximum packing set of vertices of $G$ and assign the color $1$ to each vertex of $U$. Assigning distinct colors different from 1 to each vertex of $V(G)\setminus U$ produces a harmonious coloring of $G$. Hence
$h(G) \leq |V(G)\setminus U | + 1 = n -\rho (G) + 1.$
\end{proof}

Observe that a harmonious homomorphism $f$ of a graph $G$ such that $f$ identify a set of vertices in a maximum packing set produces a graph $f(G)$ of diameter at most 4.

%%%%%%%%%%%%%%%%%%%%%%%%%%%%%%%%%%%%%%%%%%%%%%%%%%%%%%%%%%%%%%%%%%

\section{On circulant graphs}\label{section4}

Let $J \subseteq \{1, \ldots, \lfloor\frac{n}{2}\rfloor\}\subseteq \mathbb{Z}_n$, where $n \geq 3$. The \emph{circulant graph} $C_n(J)$ is defined by $V(C_n(J)) = \mathbb{Z}_n$ and $E(C_n(J)) = \{ij \colon |\{i-j,j-i\}\cap J| = 1\}$; $J$ is called the set of jumps of $C_n(J)$.  Complete graphs and cycles are examples of circulant graphs. Observe that the number of edges in $C_n(J)$ is either $n|J|-\frac{n}{2}$ if $n$ is even and $\frac{n}{2}\in J$, or $n|J|$ otherwise.

In \cite{MR3729832}, D{\k e}bsky et al.  proved that for a fixed set $J$, $h(C_n(J))\leq \sqrt{2n|J|}+O(n^{0.25}(\log n)^{0.5}).$ In other words,  if the circulant graph has small regularity, then its harmonious chromatic number is asymptotically close to the lower bound of Equation \ref{eq1}. 

In Theorem \ref{teo13} we present a family of circulant graphs with two jumps and a harmonious chromatic number equal to the lower bound of Equation \ref{eq1}. In Theorem \ref{teo16}, we give a family of circulant graphs of order $n$ with regularity $\Omega(\sqrt{n})$ and a harmonious chromatic number equal to $n$, where $n$ follows the asymptotic distribution of the prime numbers.

Given a graph $G$, the graph $G^{\bowtie}$ is defined as follows, see \cite{MR3946682}. 

If $V (G) = \{w_1, w_2,\dots, w_n\}$, then $V(G^{\bowtie}) = \{u_1,u_2,\dots ,u_n,v_1,v_2,\dots ,v_n\}$, and vertices $x_i,y_j$, where $x,y \in \{u,v\}$ and $i,j \in \{1,2,\dots ,n\}$, are adjacent in $G^{\bowtie}$ if and only if $w_iw_j \in E (G)$. For instance, if $w_1w_2$ is an edge of a graph $G$, the graph $G^{\bowtie}$ has the edges $u_1u_2$, $v_1v_2$, $u_1v_2$ and $v_1u_2$, see Figure \ref{Fig1}.
\begin{figure}[htbp]
\begin{center}	
\includegraphics[scale=0.9]{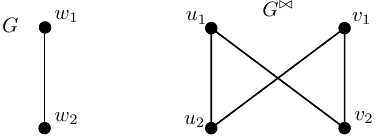}
\caption{\label{Fig1} An edge of $G$ produces four edges in $G^{\bowtie}$.}
\end{center}
\end{figure}

\begin{theorem}\label{teo13}
The circulant graph $C_{t(t-1)}(1,\frac{t(t-1)}{2}-1)$ for $t\equiv -1 ( \mod 4)$ has the harmonious chromatic number equal to the lower bound of Equation \ref{eq1}.
\end{theorem}
\begin{proof}
Since $t$ is odd, the complete graph $K_t$ is Eulerian. Suppose that $V (K_t) = \{x_1, x_2, \dots , x_t\}$. The Eulerian circuit of $K_t$ presented as a cycle $C_{\frac{t(t-1)}{2}}$ with vertex set $\{w_0, w_1, \dots , w_{\frac{t(t-1)}{2}} \}$, in which each vertex of $K_t$ appears $t-1$ times, can be unfolded as $C^{\bowtie}_{\frac{t(t-1)}{2}}$. The situation is illustrated in Figure \ref{Fig2} (left), where $u$-vertices ($v$-vertices) form the interior (exterior) cycle $C_{int} (C_{ext} )$ of length $\frac{t(t-1)}{2}$ of the graph $C^{\bowtie}_{\frac{t(t-1)}{2}}$, and $C_{int}$ corresponds to the Eulerian circuit of $K_t$. From $t\equiv -1 (\mod 4)$ it follows that $\frac{t(t-1)}{2}$ is odd. Therefore, the edges joining $C_{int}$ and $C_{ext}$ form the cycle of length $t(t-1)$ with vertices $z_0,z_1,\dots ,z_{t(t-1)}$ as in Figure \ref{Fig2} (right). Note that $z$-vertices with even (odd) subscripts are those of $C_{int} (C_{ext} )$, and the graph $C^{\bowtie}_{\frac{t(t-1)}{2}}$ is isomorphic to the circulant graph $G = C_{t(t-1)}(1, \frac{t(t-1)}{2} - 1)$. 
\begin{figure}[htbp]
\begin{center}	
\includegraphics[scale=0.9]{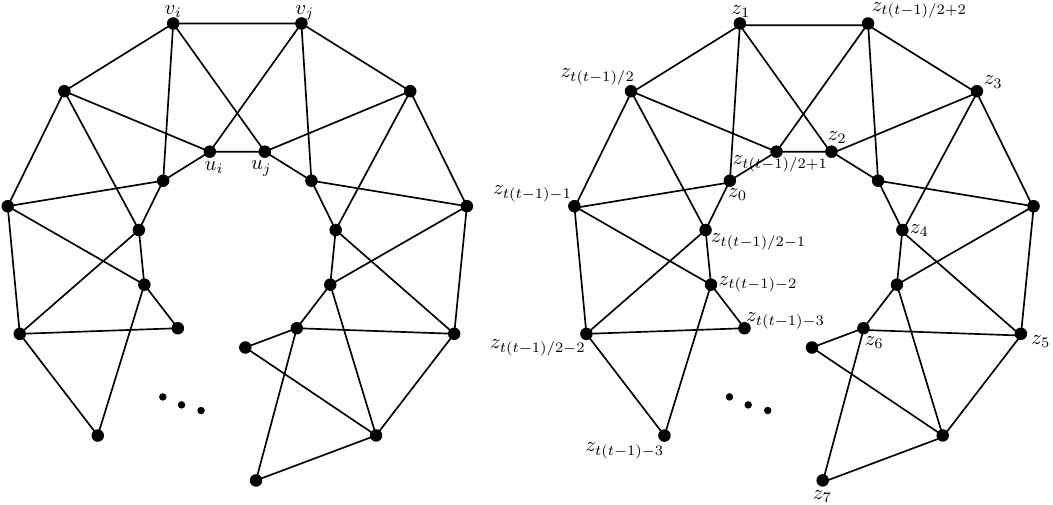}
\caption{\label{Fig2} The graph $C_\frac{t(t-1)}{2}^{\bowtie}$.}
\end{center}
\end{figure}

The graph $G$ is colored with colors $x_1,x_2,\dots ,x_t,y_1,y_2,\dots ,y_t$, where a color $y_l, l \in\{ 1, 2, \dots , t\}$, corresponds to the color $x_l$ in the projection of $C_{int}$ from the center of $C_{int}$ onto $C_{ext}$. Clearly, each edge joining colors $x_i$ and $x_j$, $i\not = j$, gives rise to edges joining three other pairs of colors, namely $x_i$ and $y_j$, $y_i$ and $x_j$, $y_i$ and $y_j$. The coloring of $C_{int}$ is harmonious, hence the total number of distinct pairs of color classes joined by an edge is $4\frac{t(t-1)}{2} = 2t(t-1) = |E(G)|$. This means that the coloring of $G$ is harmonious, and so $h(G) \leq 2t$. On the other hand,
\[\frac{1}{2}+ \sqrt{\frac{1}{4}+2(2t(t-1))} = \frac{1}{2}+ \sqrt{\frac{1}{4}+2|E(G)|} \leq h(G),\]

since $2t-1<\frac{1}{2}+ \sqrt{4t^2-4t+\frac{1}{4}}$ (a consequence of $t \geq 3)$, and the result follows.
\end{proof}

\begin{theorem}
If a graph $G$ has size $m\geq 1$ and its harmonious number $h$ equals the lower bound $\frac{1}{2}+ \sqrt{\frac{1}{4}+2m}$, then $G^{\bowtie}$ has the harmonious number $2h$. 
\end{theorem}
\begin{proof}
From $h(G) = \frac{1}{2}+ \sqrt{\frac{1}{4}+2m}$ we obtain $m=\binom{h}{2}$. Suppose that $V(G) = \{w_1,w_2,\dots ,w_n\}$, and let $\varsigma \colon V(G) \rightarrow \{1,2,\dots ,h\}$ be a harmonious coloring of $G$. Then it is straightforward to see that the coloring $\varsigma'\colon  V(G^{\bowtie}) \rightarrow \{1,2,\dots ,2h\}$ defined by $\varsigma'(u_i) = \varsigma (w_i)$ and $\varsigma'(v_i) = h+\varsigma (w_i)$, $i\in\{1, 2, \dots , n\}$, is harmonious, too. Therefore, $h(G^{\bowtie}) \leq 2h$.

The size of the graph $G^{\bowtie}$ is $M =4m=2h(h-1)$. From $m\geq 1$ it follows that $h>1$, 
$2h-1<\frac{1}{2}+ \sqrt{\frac{1}{4}+4h(h-1)}\leq h(G^{\bowtie})$ and $h(G^{\bowtie})\geq 2h$.
\end{proof}

A natural question is the following: Which are the circulant graphs of diameter two? The obvious answer is that they are those such that any number from 1 to $n$ can be represented as the sum of two `jumps', so we suppose that the question is rather when that can be done.

Graphs of order $n$ for which $\deg(u)+\deg(v)\geq n-1$ for every two nonadjacent vertices $u,v$ have diameter at most 2,  thus circulant graphs which are $r$-regular of order $n$ with $r\geq (n-1)/2$ have harmonious number $n$. We give a construction of circulant graphs with harmonious number $n$, such that each vertex has degree $\Theta(\sqrt{n})$. 

A subset $D=\{d_1,\dots,d_k\}$ of (the underlying set of) an additive group $\Gamma$, of order $n$, is an \emph{$(n,k)$-difference set} if for each nonzero element $g\in \Gamma$, there is a unique pair of different elements $d_i,d_j\in D$ such that $g=d_i-d_j$.  Therefore, if $D\subset \Gamma$ is an $(n,k)$-difference set and $g\in \Gamma$, then $D+g:=\{d_1+g,d_2+g,\dots,d_k+g\}$ is an $(n,k)$-difference set. 

The parameters $n$ and $k$ of an $(n,k)$-difference set satisfy the necessary conditions $n-1=k(k-1)$. It is conjectured that $q=k-1$ must be a power of a prime number, and the conjecture is verified for $q\leq 2,000,000$. Then $n=q^2+q+1$ and the construction of $(q^2+q+1,q+1)$-difference sets has an algebraic approach via the Galois field of order $q$.  See Table \ref{tab1} for some examples of difference sets. See \cite{JPS2007DS} for further reading.

\begin{table}[!htbp]
\begin{center}
\begin{tabular}{|c|c|c|c|}
\hline
$\mathbb{Z}_n$ & Difference set & $\mathbb{Z}_n$ & Difference set \\
\hline
$\mathbb{Z}_{7}$ & 1,2,4 & $\mathbb{Z}_{57}$ & 1,2,4,14,33,37,44,53 \\
\hline
$\mathbb{Z}_{13}$ & 1,2,5,7 & $\mathbb{Z}_{73}$ & 1,2,4,8,16,32,37,55,64 \\
\hline
$\mathbb{Z}_{21}$ & 1,4,5,10,12 & $\mathbb{Z}_{91}$ & 1,2,4,10,28,50,57,62,78,82 \\
\hline
$\mathbb{Z}_{31}$ & 1,2,4,9,13,19 & $\mathbb{Z}_{133}$ & 1,2,4,13,21,35,39,82,89,95,105,110 \\
\hline
\end{tabular}
\caption{\label{tab1}Some difference sets for $\mathbb{Z}_n$ with small $n$.}
\end{center}
\end{table}

Let $D\subseteq \mathbb{Z}_n$ be a difference set such that $0\notin D$, and let \[J(D) = \{i \colon i \in D, i \leq n/2\} \cup \{n-i\colon i \in D, i > n/2\}.\]

Figure \ref{Fig3} shows the circulant graph $C_{31}(J)$ of diameter two with $J =\{1,2,4,9, 12,13\}$, where $J = J(D)$ and $D = \{1,2,4,9,13,19\}$.
\begin{figure}[htbp]
\begin{center}	
\includegraphics{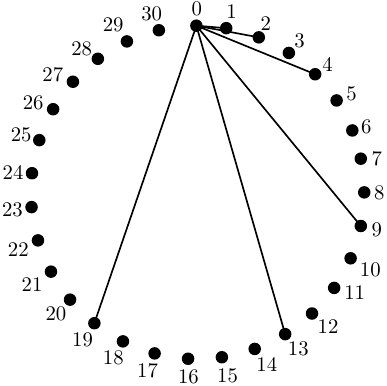}
\caption{\label{Fig3} The graph $C_{31}(1,2,4,9,12,13)$, where only the edges $0i$ for $i\in \{1,2,4,9,13,19\}$ are shown.}
\end{center}
\end{figure}

\begin{theorem}\label{teo16}
Let $q$ be a prime power, $n=q^2+q+1$, and let $D\subseteq \mathbb{Z}_n$ be a  difference set such that $0\notin D$ and $|D|=q+1$. Then $h(C_n(J(D))=n$.
\end{theorem}
\begin{proof}
The graph $G = C_n(J(D))$ is $r$-regular with $r = 2(q + 1) = 2\lceil\sqrt{n}\rceil = \Theta (\sqrt{n})$, and its diameter is two. To see it consider distinct vertices $u,v$ of $G$. There exist $i,j \in D$ such that $i-j = u-v = l$. Observe that $|J(D)\cap \{i, n-i\}| = |J(D)\cap \{j, n-j\}| = 1$. The vertex $w = u+(n-i) = u-i = u-j-l$ is adjacent in $G$ to both $u$ and $v$, since $v = u-l = w+j = w-(n-j)$. Thus, the distance in $G$ between $u$ and $v$ is at most two, and we are done. 
\end{proof}

\section{Conclusions}

Harmonious coloring was first proposed in 1982 \cite{MR683991}, but still very little is known about it.  For future work, we propose to study either the harmonious achromatic number defined in Section \ref{section2} or the greedy coloring via harmonious homomorphisms. We define the \emph{harmonious Grundy number} of a graph $G$ as the number of colors in the worst case using such greedy colorings from $G$ to a graph of diameter two.

Theorem \ref{teo9} estimates the harmonious chromatic number of the $(q+1,6)$-cages, hence another possible problem is to estimate the harmonious chromatic number of the $(q+1,8)$-cages, i.e., graphs that are the incidence graph of generalized quadrangles.

\begin{figure}[htbp]
\begin{center}	
\includegraphics{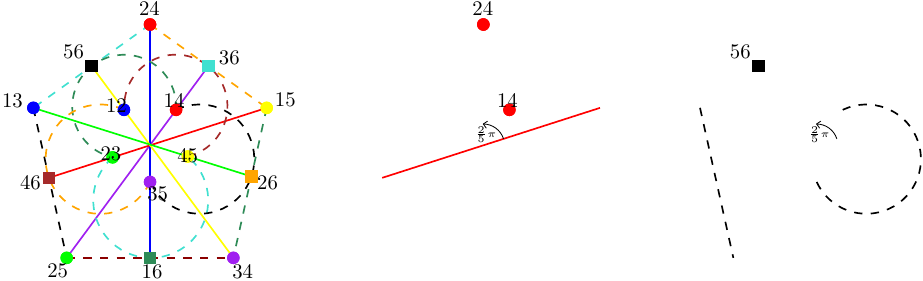}
\caption{\label{Fig4} (Left) The harmonious chromatic number of the incidence graph of $GQ(2,2)$ is 10. (Center) A color class with two points and a line produces other color classes via rotations of $\frac{2}{5}\pi$.  (Right) A color class with two lines and a point produces other color classes via rotations of $\frac{2}{5}\pi$.}
\end{center}
\end{figure}

The smallest non-trivial generalized quadrangle is $GQ(2,2)$, whose representation can be obtained through a 1-factorization of $K_6$ as follows, see \cite{web}. If $V(K_6)=\{1,\dots,6\}$, the points of $GQ(2,2)$ are the edges of $K_6$ and each line has three points arising from a matching of $K_6$.  In Figure \ref{Fig4} (Left), $GQ(2,2)$ is shown with a harmonious coloring of 10 colors where each color class has three elements: There are two types of color classes, the first one contains two points and a line (as in Figure \ref{Fig4} (Center)), while the second one contains two lines and a point (as in Figure \ref{Fig4} (Right)). The mentioned two color classes give rise to the remaining eight color classes by rotating counterclockwise around the center of the pentagon with vertices $13$, $24$, $15$, $34$, $25$ through the angle $\frac{2}{5}\pi j$, $j\in\{1,2,3,4\}$. Therefore, since the $(3,8)$-cage graph $G$ is the incidence graph of $GQ(2,2)$ and it has $45$ edges, it requires at least 10 colors in any harmonious coloring, therefore $h(G)=10$. As consequence, the achromatic number of $G$ is also 10.

%%%%%%%%%%%%%%%%%%%%%%%%%%%%%%%%%%%%%%%%%%%%%%%%%%%%%%%%%%%%%%%%%%

\section{Statements and Declarations}

This work was partially supported by PAPIIT-M{\' e}xico under Project IN108121; CONACyT-M{\' e}xico under projects 282280,  A1-S-12891, 47510664; and PAIDI-M{\' e}xico under Project PAIDI/007/21.

The authors have no relevant financial or non-financial interests to disclose.

All authors contributed to the study conception and design. The first draft of the manuscript was written by Gabriela Araujo-Pardo, Juan Jos{\' e} Montellano-Ballesteros, Christian Rubio-Montiel, and Mika Olsen, and all authors commented on previous versions of the manuscript. All authors read and approved the final manuscript.

The authors wish to thank the anonymous referees for their suggestions and remarks to improve this paper.

%---------------------- Bibliography ---------------------

\bibliographystyle{plain}
\bibliography{biblio}

\end{document}